\documentclass[12pt,twoside]{amsart}
\usepackage{graphicx}
\usepackage{tikz}
\usetikzlibrary{decorations.markings}
\usetikzlibrary{arrows,shapes,positioning}
\usetikzlibrary{topaths}
\tikzstyle{vertex}=[circle, draw, inner sep=0pt, minimum size=4pt]
\usetikzlibrary{positioning}
\usepackage{float}
\newcommand{\vertex}{\node[vertex]}

\usepackage{amssymb,amscd}
\usepackage{mathrsfs}
\usepackage{array,float}
\usepackage[all]{xy}
\usepackage{enumerate}

\theoremstyle{plain}
\newtheorem{thm}{Theorem}[section]

\newtheorem{lem}[thm]{Lemma}
\newtheorem{pro}[thm]{Proposition}
\newtheorem{cor}[thm]{Corollary}
\newtheorem{conjecture}[thm]{Conjecture}

\theoremstyle{remark}

\theoremstyle{definition} 
\newtheorem{definition}[thm]{Definition}

\numberwithin{equation}{section}

\newcommand{\Z}{\mathbb{Z}}   
\newcommand{\Q}{\mathbb{Q}}

\newcommand{\F}{\mathbb{F}}

\newcommand{\ackn}{  \noindent{\sc Acknowledgement }\hspace{5pt} }

\renewcommand{\phi}{\varphi}
\def\G{{\mathcal G}}
\def\H{{\mathcal H}}

\begin{document}

\author{Ilir Snopce}
\address{Universidade Federal do Rio de Janeiro\\
  Instituto de Matem\'atica \\
  21941-909 Rio de Janeiro, RJ \\ Brazil }
\email{ilir@im.ufrj.br}

\author{Pavel Zalesskii}
\address{University of Bras\'ilia\\
  Department of  Mathematics \\
  70910-9000, Bras\'ilia \\ Brazil }
\email{pz@mat.unb.br}

\thanks{}

\title[Right-angled Artin pro-$p$ groups] {Right-angled Artin pro-$p$ groups}

\begin{abstract} 
Let $p$ be a prime. The right-angled Artin pro-$p$ group $G_{\Gamma}$  associated  to a fnite simplicial graph  $\Gamma$ is the pro-$p$ completion of the right-angled Artin group associated to $\Gamma$. We prove that the following assertions are equivalent:
(i) no induced subgraph of $\Gamma$ is a square or a line with four vertices (a path of length 3);
(ii) every closed subgroup of  $G_{\Gamma}$ is itself a right-angled Artin pro-$p$ group (possibly infinitely generated); 
(iii) $G_{\Gamma}$ is a Bloch-Kato pro-$p$ group;
(iv) every closed subgroup of  $G_{\Gamma}$ has torsion free abelianization;
(v) $G_{\Gamma}$ occurs as the maximal  pro-$p$ Galois group $G_K(p)$ of some field $K$ containing a primitive $p$th root of unity;
 (vi) $G_{\Gamma}$ can be constructed from $\Z_p$  by iterating two group theoretic operations, namely, direct products with $\Z_p$ and  free pro-$p$ products. 
  This settles in the affirmative  a conjecture of Quadrelli and Weigel. Also, we show that the Smoothness Conjecture of De Clercq and Florens holds for right-angled Artin pro-$p$ groups. Moreover, we prove that  $G_{\Gamma}$ is coherent if and only if each circuit of  $\Gamma$ of length greater than three has a chord.

\end{abstract}


\maketitle
\section{Introduction}
 Throughout the paper $p$ denotes a
prime. Let $K$ be a field. The absolute Galois group of $K$ is the profinite group $G_K= \textrm{Gal}(K_s/K)$, where $ K_s$ is a separable closure of $K$. The maximal pro-$p$ Galois group of $K$, denoted by $G_K(p)$, is the maximal pro-$p$ quotient of $G_K$.  More precisely, $G_K(p) =  \textrm{Gal}(K(p)/K)$, where  $K(p)$ is 
the composite of all finite Galois $p$-extensions of $K$ (inside $K_s$). Describing absolute Galois groups of fields among profinite groups is one of the most important problems in Galois theory. Already describing  $G_K(p)$  among pro-$p$ groups is a highly non-trivial task; there has been an intensive study in this direction during the last few decades. Efrat conjectured that if  $G_K(p)$ is finitely generated, then it can be built from $\Z_p$, Demushkin groups and $\Z / 2\Z$ if $p=2$,  by iterating two group theoretic operations, namely, free  pro-$p$ products and certain semidirect products; this is the so-called elementary type conjecture (cf. \cite{Ef97} and  \cite{Ef98}). Pro-$p$ groups that can be obtained as in the elementary type conjecture are said to be of \emph{elementary type}.

\smallskip

A particular challenge is to determine what special properties $G_K(p)$ have among all pro-$p$ groups; so far only a few properties have been found (cf. for example \cite{BeNiJaJo07}, \cite{MiTa16}, \cite{EfQu19},  \cite{QuWe18}  and references therein).  A pro-$p$ group $G$ is called \emph{$H^\bullet$-quadratic} (or simply \emph{quadratic}) if the  the graded algebra $H^\bullet(G,\F_p) = \bigoplus_{n\geq0}H^n(G,\F_p)$, endowed with the cup product and with $\F_p$ as a trivial $G$-module,
is a \emph{quadratic algebra} over $\F_p$, i.e., all its elements of positive degree are combinations
of products of elements of degree 1, and its defining relations are homogeneous relations of degree 2 (see \cite{QuSnVa19}).  A pro-$p$ group $G$ is called a \emph{Bloch-Kato pro-$p$ group} if  every closed subgroup $U$ of $G$ is quadratic (see \cite{Qu14}).  From the positive solution of the Bloch-Kato conjecture by Rost and Voevodsky with a `patch' of Weibel (cf.  \cite{Ro02},  \cite{Vo11} and  \cite{We09}) it follows that the maximal pro-$p$ Galois group $G_K(p)$ of a field $K$ that contains a primitive $p$th root of unity is a Bloch-Kato pro-$p$ group. It is remarkable that we do not know a single example of a Bloch-Kato pro-$p$ group which is not of elementary type.

\smallskip

 Let $\Gamma$ be a simplicial graph with vertex set $V(\Gamma)$ and edge set $E(\Gamma)$. The \emph{right-angled Artin group}  $G(\Gamma)$ associated to $\Gamma$ is the group with generating set $V(\Gamma)$, with the relation that vertices  $u$ and  $v$ commute imposed  whenever $u$ and  $v$ span an edge of  $\Gamma$. Right-angled Artin groups play a central role in geometric group theory (see \cite{Cha07} and references therein). Fundamental groups of hyperbolic 3-manifolds (\cite{Ag13}, \cite{KaMa12} and \cite{Wi11}), 1-relator groups with torsion (\cite{Wi11}) and many small cancellation groups (\cite{Ag13} and \cite{Wi11}) are virtually special groups, and every  special group embeds in a right-angled Artin group (\cite{HaWi08}). 
   
    If $\Gamma$ is finite, then the \emph{right-angled Artin pro-$p$ group} $G_{\Gamma}$ associated to $\Gamma$ is defined to be the pro-$p$ completion of $G(\Gamma)$. Thus   $G_{\Gamma}$   is the pro-$p$ group defined by the following pro-$p$ presentation: 
$$ G_{\Gamma} = \langle  V(\Gamma) ~ \mid ~ [u, v]=1  \textrm{ if and only if }   \{ u, v \}   \in E( \Gamma) \rangle.$$ 
An infinitely generated \emph{right-angled Artin pro-$p$ group}  is defined by the same pro-$p$ presentation for an infinite profinite graph $\Gamma$ without loops (see  Section \ref{preliminaries}).
 Kropholler and Wilkes proved that right-angled Artin groups associated to finite graphs are determined by their pro-$p$ completions (see  \cite[Theorem~4]{KroWi16}). This shows that the class of right-angled Artin  pro-$p$ groups  is at least as rich as the class of the abstract ones.

\smallskip
 Kim and Roush proved that the $\F_p$-cohomology
 algebra of  a  right-angled Artin group  $G(\Gamma)$ associated to a finite graph  is quadratic ( \cite{KiRo80}), while Lorensen proved that    the $\F_p$-cohomology
 algebra of   $G(\Gamma)$  coincides with that of its pro-$p$ completion $G_{\Gamma}$ (\cite{Lo10}). Hence,  right-angled Artin pro-$p$ groups associated to finite graphs are quadratic. Therefore the question of characterizing Bloch-Kato pro-$p$ groups among the right-angled Artin pro-$p$ groups appears naturally.  This question has attracted experts in the field, since this could provide us a possible candidate for a Bloch-Kato pro-$p$ group which is not of  elementary type.  The following conjecture is due to Quadrelli and Weigel.

\begin{conjecture}\label{conj-Bloch}
Let $\Gamma$ be a finite simplicial graph and let $G = G_{\Gamma}$ be the  right-angled Artin pro-$p$ group associated to $\Gamma$. Then $G$ is a Bloch-Kato pro-$p$ group if and only if
no induced subgraph of $\Gamma$ has either of the two forms
\begin{figure}[H]
\[\begin{tikzpicture}[x=.3 cm, y=.55 cm, 
    every edge/.style={
        draw,
        postaction={decorate,
                    decoration=
                   }
        }
]
	\vertex[fill] (v_0) at (0, 0){};
	\vertex[fill] (v_1) at (9.00, 0){};
        \vertex[fill] (v_2) at (0, 3.00){};	
        \vertex[fill] (v_3) at (9.00, 3.00){};	
        \vertex[fill] (v_4) at (16.00, 2.00){};	
         \vertex[fill] (v_5) at (21.00, 2.00){};	
          \vertex[fill] (v_6) at (26.00, 2.00){};	
           \vertex[fill] (v_7) at (31.00, 2.00){};

\path

        (v_0) edge node[below]{$C_4$} (v_1)
	(v_0) edge  (v_2)
	(v_2) edge  (v_3) 
	(v_1) edge  (v_3)
	(v_4) edge  (v_5)
	 (v_5) edge node[below]{$L_3$} (v_6)
	(v_6) edge  (v_7)
        ;
\end{tikzpicture}\]
\end{figure}

\end{conjecture}

By  \cite[Theorem~5.6]{Qu14}, the right-angled Artin pro-$p$ group $G_{C_4}$ is not Bloch-Kato; in particular, it does not occur as  the maximal  pro-$p$ Galois group $G_K(p)$ of some field $K$ containing a primitive $p$th root of unity. Quadrelli and Weigel conjectured that $G_{L_3}$ does not occur as  $G_K(p)$ as well (cf.  \cite{CaQu19}); clearly, this would follow  from Conjecture \ref{conj-Bloch}.

\smallskip

The main purpose of this paper is to show that all these conjectures hold. Before stating our main result, we will discuss one more property of a $p$-Sylow subgroup of the absolute Galois group of a field containing all $p$-power roots of unity.  A pro-$p$ group $G$ is called \emph{absolutely torsion free} if for every closed subgroup $H$ of $G$ the abelianization $H^{ab}$ is torsion free.  This property was introduced and studied by W\"urfel in \cite{Wu85}. Free pro-$p$ groups, free abelian pro-$p$ groups and pro-$p$ completions of surface groups are examples of absolutely torsion free pro-$p$ groups. W\"urfel proved that the class of absolutely torsion free pro-$p$ groups is closed under forming inverse limits and free pro-$p$ products (see  \cite[Proposition 3]{Wu85}); moreover, he proved that a direct product of a free abelian pro-$p$ group and an absolutely torsion free pro-$p$ group is absolutely torsion free. One may ask are there any other pro-$p$ groups that are absolutely torsion free. Since the abelianization of every right-angled Artin  pro-$p$ group is torsion free, it is a natural problem to determine which of these groups are absolutely torsion free.

 \smallskip

Finally, we are ready to state the main result of our paper.

\begin{thm}\label{thm:main}
Let $\Gamma$ be a finite simplicial graph and let $G = G_{\Gamma}$ be the  right-angled Artin pro-$p$ group associated to $\Gamma$. Then the following assertions are equivalent:
\begin{itemize}
\item [(i)] No induced subgraph of $\Gamma$ has either of the two forms
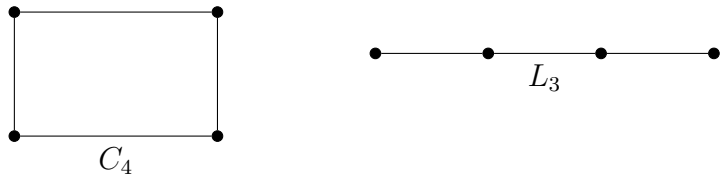
\begin{figure}[H]
\[\begin{tikzpicture}[x=.3 cm, y=.55 cm, 
    every edge/.style={
        draw,
        postaction={decorate,
                    decoration=
                   }
        }
]
	\vertex[fill] (v_0) at (0, 0){};
	\vertex[fill] (v_1) at (9.00, 0){};
        \vertex[fill] (v_2) at (0, 3.00){};	
        \vertex[fill] (v_3) at (9.00, 3.00){};	
        \vertex[fill] (v_4) at (16.00, 2.00){};	
         \vertex[fill] (v_5) at (21.00, 2.00){};	
          \vertex[fill] (v_6) at (26.00, 2.00){};	
           \vertex[fill] (v_7) at (31.00, 2.00){};

\path

        (v_0) edge node[below]{$C_4$} (v_1)
	(v_0) edge  (v_2)
	(v_2) edge  (v_3) 
	(v_1) edge  (v_3)
	(v_4) edge  (v_5)
	 (v_5) edge node[below]{$L_3$} (v_6)
	(v_6) edge  (v_7)
        ;
\end{tikzpicture}\]
\end{figure}
\item [(ii)] Every closed subgroup of  $G$ is itself a right-angled Artin pro-$p$ group.
\item [(iii)] $G$ is an absolutely torsion free pro-$p$ group.
\item [(iv)] $G$ is a Bloch-Kato pro-$p$ group.
\item [(v)] $G$ occurs as the maximal  pro-$p$ Galois group $G_K(p)$ of some field $K$ containing a primitive $p$th root of unity.
\item [(vi)]  $G$ can be constructed from $\Z_p$  by iterating two group theoretic operations: direct products with $\Z_p$ and  free pro-$p$ products. 

\end{itemize}
\end{thm}

Recently, in \cite{MiPaPaTa18}, Min\'a\v{c} et al.  conjectured that the  $\F_p$-cohomology algebra $H^\bullet(G_K(p),\F_p) = \bigoplus_{n\geq0}H^n(G_K(p),\F_p)$ of $G_K(p)$ is universally Koszul (see \cite{MiPaPaTa18}  for the definition and related notions), whenever  $K$ is a field containing a primitive $p$th root of unity. Universal Koszulity is a very strong property.    Cassella and Quadrelli proved that the  $\F_p$-cohomology algebra $H^\bullet(G_{\Gamma},\F_p)$ of a right-angled Artin pro-$p$ group associated to a finite graph $\Gamma$ is universally Koszul if and only if no induced subgraph of $\Gamma$ is a square $C_4$ or a line with four vertices $L_3$ (cf. \cite[Theorem 5.2]{CaQu19}).  This  result together with Theorem \ref{thm:main} imply that one can not find a counterexample to the conjecture of Min\'a\v{c} et al.  among right-angled Artin pro-$p$ groups.

\begin{cor}
Let $\Gamma$ be a finite simplicial graph. Then $G_{\Gamma}$ occurs as the maximal  pro-$p$ Galois group $G_K(p)$ of some field $K$ containing a primitive $p$th root of unity if and only if the $\F_p$-cohomology algebra $H^\bullet(G_{\Gamma},\F_p)$  is universally Koszul.
\end{cor}

A  \emph{$p$-oriented profinite group} is a profinite  group $G$ with a continuous homomorphism $ \theta: G  \to  \Z_p^{ \times}$, where $ \Z_p^{ \times}$ is the group of units of $ \Z_p.$  Given a $p$-oriented profinite group $(G,   \theta)$, let  $\Z_p(1)$ be the  (left) $ \Z_p{ [\![G ]\!]}$-module defined in the following way: $\Z_p(1)$ is equal to the additive group of $ \Z_p$ and the $G$-action is given by
 \begin{equation}\label{action}
  g  \cdot z =  {\theta(g)}  \cdot z,  \textrm{ for } g \in G \textrm{ and } z  \in \Z_p(1).
  \end{equation}
  Here  $ \Z_p{ [\![G ]\!]}$ denotes the completed $ \Z_p$-group algebra of $G$. 
  

 A  $p$-oriented profinite group $(G,   \theta)$ is called $1$- \emph{smooth} ($1$-\emph{cyclotomic} in  \cite{QuWe18}), if 
  it satisfies  one of the following equivalent conditions (see  \cite[Def. 14.18]{CleFlo17} and  \cite[Fact 2.1]{QuWe18} ):
  \begin{itemize}
  \item [(i)] $H^{2}_{cts}(U,\Z_p(1))$ is a torsion free abelian group for every open subgroup $U$ of $G$;
  \item [(ii)] the natural epimorphism $\Z_p(1)  \twoheadrightarrow {\Z_p(1)}/p{\Z_p(1)}$ induces an epimorphism 
  $$H^{1}_{cts}(U,\Z_p(1)) \twoheadrightarrow H^{1}(U, {\Z_p(1)}/p{\Z_p(1)})$$
  for every open subgroup $U$ of $G.$
   \end{itemize}
 Here $H^{1}_{cts}(G, -)$ denotes continuous co-chain cohomology as introduced by Tate in \cite{Ta76}.

In \cite{CleFlo17}, De Clercq and Florens stated the following conjecture, which is known as the Smoothness Conjecture.

\begin{conjecture} \cite[Conj. 14.25]{CleFlo17}\label{smoothness-conj}
If a profinite group $G$ can be equipped with a continuous homomorphism $ \theta: G  \to  \Z_p^{ \times}$ such that $(G,  \theta)$ is $1$-smooth, then $G$ has the weak Bloch-Kato property. 
\end{conjecture}

As claimed by the authors, the work in  \cite{CleFlo17} is motivated by the search for an ``explicit'' proof of the Bloch-Kato conjecture in Galois cohomology. For the definition of the weak Bloch-Kato property see \cite[Def. 14.23]{CleFlo17}. Here we only remark that every Bloch-Kato pro-$p$ group has the weak Bloch-Kato property. Our next result, which essentially follows from Theorem \ref{thm:main}, shows that the Smoothness Conjecture holds for right-angled Artin pro-$p$ groups.

\begin{thm}\label{smoothness-thm}
Let $\Gamma$ be a finite simplicial graph and let $G = G_{\Gamma}$ be the  right-angled Artin pro-$p$ group associated to $\Gamma$. Then $G$ can be equipped with a continuous homomorphism $ \theta: G  \to  \Z_p^{ \times}$ such that $(G,  \theta)$ is $1$-smooth   if and only if $G$ is a Bloch-Kato pro-$p$ group.
\end{thm}

Recall that a pro-$p$ group $G$ is  of \emph{(homological)  type}  $FP_{\infty}$ if the trivial pro-$p$  $\Z_p[[G]]$-module $\Z_p$ has a
 pro-$p$ projective resolution with all $\Z_p[[G]]$-modules finitely generated; equivalently, if the Galois cohomology groups $H^i(G, \F_p)$ are finite for each $i \geq 0$ (see \cite[Theorem1.6]{Ple80}). A pro-$p$ group $G$ is called  \emph{coherent} if each of its finitely generated closed subgroups is finitely presented. Our final result characterizes  right-angled Artin pro-$p$ groups with these two properties.
 
 \begin{thm}\label{coherence}
Let $\Gamma$ be a finite simplicial graph and let $G = G_{\Gamma}$ be the  right-angled Artin pro-$p$ group associated to $\Gamma$. Then the following assertions are equivalent: 
\begin{itemize}
\item [(i)] Each circuit of  $\Gamma$ of length greater than three has a chord.
\item [(ii)] $G$ is coherent.
\item [(iii)]  Every finitely generated subgroup of $G$ is of type $FP_{\infty}.$  
\end{itemize}
\end{thm}

 For (abstract) right-angled Artin groups, the equivalence of $(i)$ and $(ii)$ was proved by Droms in \cite{Dr87c}, while $(ii)$ and $(iii)$ are not equivalent.  In \cite{BeBra07},  Bestvina and Brady gave examples of abstract groups of type $FP_{\infty}$  that are not finitely presented, solving a longstanding open problem;  indeed,  their examples are subgroups of some right-angled Artin groups.
 
\section{Preliminaries on Bass-Serre theory for pro-$p$ groups}\label{preliminaries}

 In this section  we collect results from the theory of pro-$p$ groups acting on pro-$p$ trees that will be used in the paper. Throughout the paper, all groups are pro-$p$ groups, subgroups are tacitly taken to be closed, homomorphisms are continuous and by generators we mean topological generators as appropriate. 
 
 \smallskip
 
We start with some definitions, following \cite{horizons}.

\smallskip

A {\it profinite graph} is a triple $(\Gamma, d_0, d_1)$, where
$\Gamma$ is a boolean space and $d_0, d_1: \Gamma \to \Gamma $ are
continuous maps such that $d_id_j=d_j$ for $i, j \in \{0, 1 \}$.
The elements of $V(\Gamma):=d_0(\Gamma)\cup d_1(\Gamma)$ are called the
{\it vertices} of $\Gamma$ and the elements of
$E(\Gamma):=\Gamma-V(\Gamma)$ are called the {\it edges} of
$\Gamma$. If $e\in E(\Gamma)$, then $d_0(e)$ and $d_1(e)$ are
called the initial and terminal vertices of $e$. If there is no
confusion, one can just write $\Gamma$ instead of $(\Gamma, d_0,
d_1)$. 

A {\em morphism} $f:\Gamma \to \Delta$  of profinite graphs is a map $f$ which
commutes with the $d_i$'s. Thus it sends vertices to vertices, but
might send an edge to a vertex. 

Every profinite graph $\Gamma$ can be represented as an inverse limit $\Gamma=\varprojlim \Gamma_i$ of its finite quotient graphs (\cite[Proposition 1.5]{horizons}).

A profinite graph $\Gamma$ is said to be {\it connected} if all
its finite quotient graphs are connected. Every profinite graph is
an abstract graph, but in general a connected profinite graph is
not necessarily connected as an abstract graph.

\medskip
\begin{definition} Let $\Gamma$ be a profinite graph without loops. A pro-$p$ group $G_{\Gamma}$ given by the pro-$p$ presentation $G_{\Gamma}=\langle v\in V(\Gamma)\mid [d_0(e),d_1(e)]=1, e\in E(\Gamma)\rangle$ is called a  \emph{right-angled Artin pro-$p$ group}. 
\end{definition}

If $\Gamma$ does not have edges, then clearly $G_{\Gamma}$ is a free pro-$p$ group $F(V(\Gamma))$ on the space of vertices. If $\Gamma$ is finite disconnected, then $G_{\Gamma}$ is a free product of right-angled Artin pro-$p$ groups associated to its connected components.

\medskip
Let $(E^*(\Gamma), *)=(\Gamma/V(\Gamma), *)$ be the pointed profinite
quotient space with $V(\Gamma)$ as  distinguished point, and let
$\mathbb{F}_p[[E^*(\Gamma), *]]$ and $\mathbb{F}_p[[V(\Gamma)]]$ be
respectively the free profinite $\mathbb{F}_p$-modules over the pointed
profinite space $(E^*(\Gamma), *)$ and over the profinite space
$V(\Gamma)$ (cf. \cite[Section5.2]{RZ10}). Note that, when
  $E(\Gamma)$ is closed, then $\mathbb{F}_p[[E^*(\Gamma), *]]=\mathbb{F}_p[[E(\Gamma)]]$. Let the maps $\delta: \mathbb{F}_p[[E^*(\Gamma), *]] \to \mathbb{F}_p[[V(\Gamma)]] $ and $\epsilon : \mathbb{F}_p[[V(\Gamma)]] \to \mathbb{F}_p$ be defined respectively by $\delta(e)=d_1(e)-d_0(e)$ for all $e\in E^*(\Gamma)$ and $\epsilon(v)=1$ for all $v\in V(\Gamma)$. Then we have the following sequence of free profinite $\mathbb{F}_p$-modules
\begin{equation*}
    \begin{CD}
      0 @>>> \mathbb{F}_p[[E^*(\Gamma), *]] @>\delta>> \mathbb{F}_p[[V(\Gamma)]] @>\epsilon>>
      \mathbb{F}_p @>>> 0.
    \end{CD}
  \end{equation*}

  The profinite graph $\Gamma$ is a {\it pro-$p$ tree} if the above sequence is exact. If $T$ is a pro-$p$ tree, then we say that a pro-$p$ group $G$ acts on $T$ if it acts continuously on $T$ and the action commutes with $d_0$ and $d_1$.
  
  \medskip
  When we say that $(\G,\Gamma)$ is a finite graph of pro-$p$ groups we mean that it contains the data of the
underlying finite graph, the edge pro-$p$ groups, the vertex pro-$p$ groups and the attaching continuous maps. More precisely,
let $\Gamma$ be a connected finite graph.  A    graph of pro-$p$ groups $(\G,\Gamma)$ over
$\Gamma$ consists of  specifying a pro-$p$ group $\G(m)$ for each $m\in \Gamma$, and continuous monomorphisms
$\partial_i: \G(e)\longrightarrow \G(d_i(e))$ for each edge $e\in E(\Gamma)$.

\smallskip
A {\em morphism} of graphs of pro-$p$ groups  $(\alpha,\bar\alpha):
(\G,\Gamma) \rightarrow (\H,\Delta)$ is a pair  of continuous
maps $\alpha:\G\longrightarrow\H$, $\bar\alpha:\Gamma\longrightarrow
\Delta$ such that $\alpha_{\G(m)}:\G(m)\longrightarrow
\H(\bar\alpha(m))$ is a homomorphism for each $m\in \Gamma$ and the following diagram commutes

$$\xymatrix{
\G\ar@{->}^\alpha[rr]\ar@{->}^{\partial_{i}}[d] & &\H\ar@{->}^{\partial_i}[d]\\
\G\ar@{->}^{\bar\alpha}[rr] & &\H }$$  We say that $(\alpha, \bar\alpha)$ is a monomorphism if both $\alpha,\bar\alpha$ are injective. In this case its image will be called a subgraph of  groups of $(\H, \Delta)$. In other words, a {\em subgraph of groups} of  a graph of pro-$p$-groups
  $(\G,\Gamma)$ is a graph of groups $(\H,\Delta)$, where $\Delta$ is a
subgraph of $\Gamma$ (i.e., $E(\Delta)\subseteq E(\Gamma)$ and
$V(\Delta)\subseteq V(\Gamma)$, the maps $d_i$ on $\Delta$ are the
restrictions of the maps $d_i$ on $\Gamma$), and for each $m\in\Delta,$
$\H(m)\leq \G(m)$.

The  fundamental pro-$p$ group
$$G= \Pi_1(\G,\Gamma)$$
of a finite graph of pro-$p$ groups $(\G,\Gamma)$ is defined by means
of a universal property: $G$ is a profinite pro-$p$ group together
with the following data  and conditions:
\begin{enumerate}
\item [(i)] a maximal subtree $D$ of $\Gamma$;

\smallskip
\item [(ii)]  a collection of continuous homomorphisms
$$\nu_m: \G(m)\longrightarrow G\quad (m\in \Gamma), $$
  and     a continuous  map
   $E(\Gamma) \longrightarrow  G$, denoted $e\mapsto t_e$  ($e\in E(\Gamma)$), such that
$t_e=1$ if $e\in E(D)$, and
$$(\nu_{d_0 (e)}\partial_0)(x)= t_e(\nu_{d_1 (e)}\partial_1)(x)t_e^{-1},\quad  \forall x\in \G(e), \ e\in E(\Gamma); $$

\smallskip
\item [(iii)]  the following universal property is satisfied:

\medskip
\noindent whenever one has the following data

\begin{itemize}
\item $H$ is a pro-$p$ group,\\
\item $\beta_m: \G(m)\longrightarrow H$, $ (m\in \Gamma)$, 
a collection of continuous homomorphisms,\\
\item a map $e\mapsto s_e$ ($e\in E(\Gamma)$)  with $s_e=1$ if
$e\in E(D)$, and\\
\item $(\beta_{d_0 (e)}\partial_0)(x)= s_e(\beta_{d_1
(e)}\partial_1)(x)s_e^{-1}, \forall x\in \G(e), \ e\in
E(\Gamma),  $\end{itemize}

\smallskip
\noindent then there exists a unique continuous homomorphism $\delta : G\longrightarrow  H$ such that $\delta(t_e)= s_e$
 $(e\in E(\Gamma))$, and for each $m\in\Gamma$ the diagram

\medskip

$$\xymatrix{&
G  \ar[dd]^\delta   \\  \G(m)  \ar[ru]^{\nu_m}
\ar[rd]_{\beta_m }\\ &H }$$

\medskip
\noindent commutes.
\end{enumerate}

\medskip
If $(\G,\Gamma)$ is a finite graph of
finitely generated pro-$p$ groups, then by a theorem of Serre (stating that every finite index subgroup of a finitely generated pro-$p$ group is open, cf. \cite[\S 4.8]{RZ10}) the  fundamental
pro-$p$ group $G=\Pi_1(\G,\Gamma)$ of
$(\G,\Gamma)$ is the pro-$p$ completion of the usual
fundamental group $\pi_1(\G,\Gamma)$ 
(cf. \cite[\S5.1]{Serre-77}).

In \cite[paragraph (3.3)]{ZM-89b},  the fundamental pro-$p$ group
 $G$ is  defined explicitly in terms of generators and relations
 associated to a chosen subtree $D$. Namely, 
 
 $$  \Pi_1(\G,\Gamma)=\langle
 \G(v), t_e\mid {\rm Rel}(\G(v)), t_e=1 \ {\rm for}\  e\in D, \partial_0(g)=t_e\partial_1(g)t_e^{-1},\  {\rm for}\ g\in \G(e)\rangle$$

  Moreover, in the same paper it is also proved 
that the definition given above is independent of the choice of
the maximal subtree $D$.\\

Associated with the graph of pro-$p$ groups $(\G, \Gamma)$ there is
a corresponding  {\it  standard pro-$p$ tree}
  $T=T(G)=\bigcup_{m\in \Gamma}
G/\G(m)$  (cf. \cite[Theorem 3.8]{ZM-89b}) .  The vertices of
$T$ are those cosets of the form
$g\G(v)$, with $v\in V(\Gamma)$
and $g\in G$; the incidence maps of $T$ are given by the formulas:

$$d_0 (g\G(e))= g\G(d_0(e))  \textrm{ and }  d_1(g\G(e))=gt_e\G(d_1(e)),  \textrm{ for } e \in E(\Gamma).  $$

 There is a natural  continuous action of
 $G$ on $T$, and clearly $ G\backslash T= \Gamma$. If $H$ is an open subgroup in $G$, then $\Delta= H\backslash T$ is finite and there exists a standard connected transversal $s:\Delta\longrightarrow T$.
 
 The following  structure theorem for open subgroups will be used in the proof of our main result.
 
 \begin{thm} (\cite[Corollary 4.5]{ZM:90})\label{subgroup theorem} Let $H$ be an open subgroup of the fundamental pro-$p$ group $\Pi_1(\G,\Gamma)$ of a finite graph of pro-$p$ groups. Then $H$ is the fundamental pro-$p$ group $H=\Pi_1(\H,\Delta)$ with $\Delta=H\backslash T$ and vertex and edge groups $\H(m)$ of $(\H,\Delta)$ are stabilizers $H_{s(m)}$, where $s:\Delta\longrightarrow T$ is a connected transversal of $\Delta$ in $T$.
 
 \end{thm}

 \section{Main results}

In this section we prove Theorem  \ref{thm:main}.

\begin{thm}\label{non-quadratic}
Let $G = G_{L_3}$ be  the right-angled Artin pro-$p$ group associated to the following graph

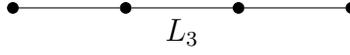
\begin{figure}[H]
\[\begin{tikzpicture}[x=.3 cm, y=.55 cm, 
    every edge/.style={
        draw,
        postaction={decorate,
                    decoration=
                   }
        }
]

        \vertex[fill] (v_4) at (16.00, 2.00){};	
         \vertex[fill] (v_5) at (21.00, 2.00){};	
          \vertex[fill] (v_6) at (26.00, 2.00){};	
           \vertex[fill] (v_7) at (31.00, 2.00){};

\path
     
(v_4) edge  (v_5)
(v_5) edge node[below]{$L_3$} (v_6)
(v_6) edge  (v_7)
        ;
\end{tikzpicture}\]
\end{figure}

Then $G$ is not a Bloch-Kato pro-$p$ group. 
\end{thm}

\begin{proof} Note that $G$ has the following pro-$p$ presentation
 $$G = \langle x, y, z, w   ~  \mid ~  [x, y] =  [y, z] =  [z, w] = 1  \rangle.$$ 
 Hence, we can consider $G$ as the fundamental pro-$p$ group of the following graph of pro-$p$ groups: 

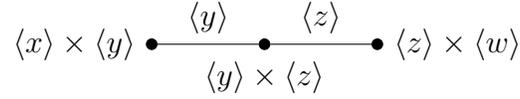
\begin{figure}[H]
\[\begin{tikzpicture}[x=.3 cm, y=.55 cm, 
    every edge/.style={
        draw,
        postaction={decorate,
                    decoration= }}]

        \vertex[fill] (v_4) at (16.00, 2.00)[label=left:$\langle x\rangle \times\langle y\rangle$]{};	
         \vertex[fill] (v_5) at (21.00, 2.00)[label=below:$\langle y\rangle \times\langle z\rangle$]{};	
          \vertex[fill] (v_6) at (26.00, 2.00)[label=right:$\langle z\rangle \times\langle w\rangle$]{};

\path
     
(v_4) edge node[above]{$\langle y \rangle$} (v_5)
(v_5) edge node[above]{$\langle z \rangle$} (v_6)

        ;
\end{tikzpicture}\]
\end{figure}

Let $C_p = \langle a \rangle$ be a cyclic group of order $p$ and  $f:G \twoheadrightarrow C_p$ be an epimorphism defined by $f(y)=f(z)=1$ and $f(x)=f(w^{-1}) = a$.  
Then $U= \textrm{ker}(f)$ is a subgroup  of $G$ of index $p$ generated by $x^p,w^p, y,z, z^{x^i}$,  and $w^ix^i=t_i$ for $i=1,\ldots p-1$, subject to the relations $[x^p,y]=[y,z]=[z,w]=1$ and $z^{t_i}=z^{x^i}$ for $i=1,\ldots p-1$,  as it follows from Theorem \ref{subgroup theorem}.  More explicitly, $U$ is the fundamental pro-$p$ group of the following graph of pro-$p$ groups.

\bigskip
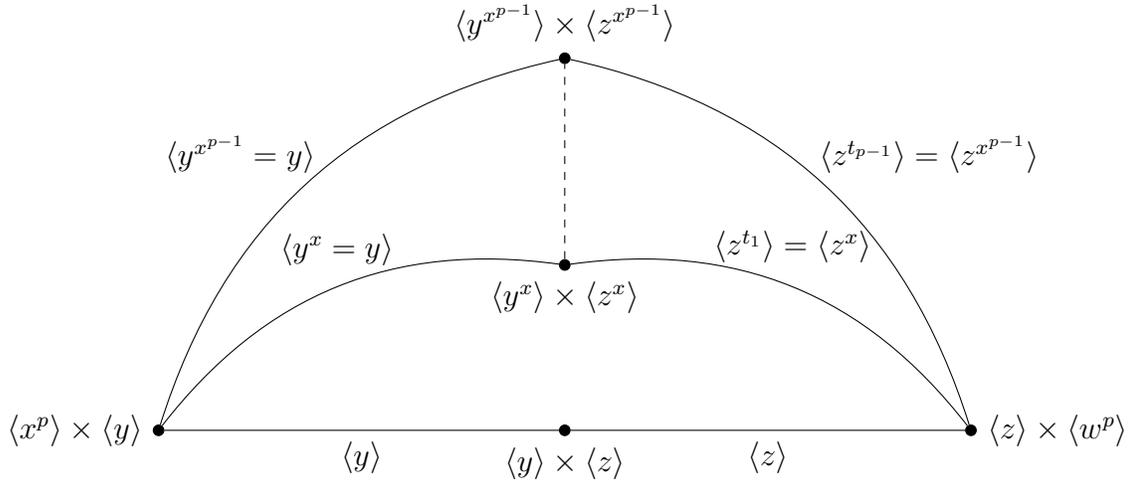
\begin{figure}[H]

\[\begin{tikzpicture}[x=.3 cm, y=.55 cm, 
every edge/.style={    draw,
    postaction={decorate,                decoration=
               }
    }
]

\vertex[fill] (v_1) at (0, 0)[label=left:$\langle x^p  \rangle  \times \langle y  \rangle  $]{};

        \vertex[fill] (v_2) at (18, 0)[label=below:$\langle y  \rangle  \times \langle z \rangle $]{};	

   	 \vertex[fill] (v_3) at (36, 0)[label=right:$\langle z  \rangle  \times \langle w^p  \rangle $]{};	

\vertex[fill] (v_4) at (18, 4)[label=below:$\langle y^x  \rangle  \times \langle z^x \rangle $]{};

 \vertex[fill] (v_5) at (18, 9)[label=above:$\langle  y^{x^{p-1}}  \rangle  \times \langle z^{x^{p-1}} \rangle $]{};

\path

(v_4) edge[dashed] (v_5)

(v_1) edge node[below]{$\langle y \rangle$} (v_2)

(v_1) edge[bend left] node[above, xshift=-25pt]{$ \langle y^{x^{p-1}}=y   \rangle$} (v_5) 

(v_4) edge[bend left] node[above, yshift=5pt]{$ \langle z^{t_1}  \rangle = \langle z^x \rangle$} (v_3) 

(v_5) edge[bend left] node[above, xshift=40pt]{$  \langle z^{t_{p-1}}  \rangle = \langle z^{x^{p-1}} \rangle$} (v_3) 

       (v_2) edge node[below]{$\langle z \rangle$} (v_3)

       (v_1) edge[bend left] node[above, yshift=4pt]{$ \langle y^x=y  \rangle$} (v_4) ;

\end{tikzpicture}\]


\end{figure}

\bigskip
Restrict this graph of pro-$p$ groups to the subgraph of groups of the lower four edges, replacing the vertex group $\langle x^p\rangle \times \langle y \rangle$ by $\langle y \rangle$ and the vertex group $\langle z\rangle \times \langle w^p \rangle$ by $\langle z \rangle$.

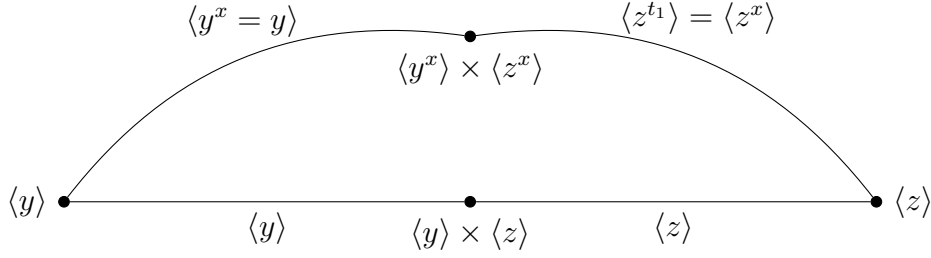
\bigskip
\begin{figure}[H]

\[\begin{tikzpicture}[x=.3 cm, y=.55 cm, 
every edge/.style={    draw,
    postaction={decorate,                decoration=
               }
    }
]

\vertex[fill] (v_1) at (0, 0)[label=left:$\langle y  \rangle  $]{};

        \vertex[fill] (v_2) at (18, 0)[label=below:$\langle y  \rangle  \times \langle z \rangle $]{};	

   	 \vertex[fill] (v_3) at (36, 0)[label=right:$\langle z  \rangle  $]{};	

\vertex[fill] (v_4) at (18, 4)[label=below:$\langle y^x  \rangle  \times \langle z^x \rangle $]{};

\path

(v_1) edge node[below]{$\langle y \rangle$} (v_2)

(v_4) edge[bend left] node[above, yshift=5pt]{$ \langle z^{t_1}  \rangle = \langle z^x \rangle$} (v_3)

       (v_2) edge node[below]{$\langle z \rangle$} (v_3)

       (v_1) edge[bend left] node[above, yshift=4pt]{$ \langle y^x=y  \rangle$} (v_4) ;

\end{tikzpicture}\]


\end{figure}

The fundamental pro-$p$ group of this subgraph is 
  a subgroup $$H=\langle y, z, z^x, t_1 = wx ~ \mid ~ [y, z]=1=[y, z^x], z^{t_1}=z^x \rangle.$$ 
  Thus       $H = \langle z, t_1, y ~ \mid ~ [z, y]=1 = [z^{t_1}, y]  \rangle.$ But  $1 = [z^{t_1}, y] = [[z,t_1], y]$ and by  \cite[Theorem~7.3]{MiPaQuTa18} (see also \cite[Proposition 2.4]{QuSnVa19}), a quadratic pro-$p$ group cannot have such a relation as a part of a minimal set of relations that define the group. Hence   $H$  is not quadratic. 
\end{proof} 

Since the maximal pro-$p$ Galois group of a field $K$ that contains a primitive $p$th root of unity is a Bloch-Kato pro-$p$ group (as it follows  from the solution of the Bloch-Kato conjecture), we deduce the following

\begin{cor}
 The right-angled Artin pro-$p$ group $G = G_{L_3}$  cannot be realized as a maximal pro-$p$ Galois group of some field $K$ that contains a primitive $p$th root of unity.
\end{cor}

\begin{thm}\label{absolutely-torsion-free} The right-angled Artin pro-$p$ group $G = G_{L_3}$  is not absolutely torsion free.

\end{thm}

\begin{proof}  Let $H$ be the subgroup of $G$ considered in the proof of Theorem \ref{non-quadratic}. Then
\[  \begin{split}
   H=\langle y, z, z^x, t = wx ~ \mid ~ [y, z]=1=[y, z^x], z^{t}=z^x \rangle \\
 =HNN(H_1=\langle y, z, z^x\rangle,z^{t}=z^x, t).
 \end{split} \]
 Note that $H_1=\langle y\rangle \times F(z,z^x)$ is a direct product of a procyclic group and a free pro-$p$ group of rank 2. Then $V_1=\langle yz, yz^x, y^p\rangle$ is a subgroup of index $p$ in $H_1$. Indeed, the natural projection $H_1\longrightarrow F(z,z^x)$ restricted to $V_1$ is surjective and so its image  in $H_1/\langle y^p\rangle\cong C_p\times F(z, z^x)$ has index $p$.  
 
 Now let $V = \langle V_1, t\rangle$. Then $V=HNN(V_1, {(z^x)}^p={(z^p)}^t, t)$ and 
$$V^{ab}=(V_1^{ab}/[t, z^p])\oplus \langle t\rangle= \Big((\langle y^p\rangle \oplus \langle yz\rangle \oplus \langle yz^x\rangle)/\langle {(z^x)}^pz^{-p}\rangle\Big) \oplus  \langle t\rangle.$$  
Since $y$ is central in $V_1$, we have $z^p=(yz)^p y^{-p}$  and ${(z^x)}^p = (yz^x)^p y^{-p}$. Thus putting $a=y^p, b=yz, c=yz^x$ and passing to the additive notation, we have  ${(z^x)}^pz^{-p} = pc-a-(pb-a)=p(c-b),$ 
and consequently,   
\[  \begin{split}
V_1^{ab}/[t, z^p]=(\langle a\rangle \oplus \langle b\rangle\oplus \langle c\rangle)/\langle p(c-b)\rangle \\
 = (\langle a\rangle \oplus \langle b\rangle\oplus \langle c-b\rangle)/\langle p(c-b)\rangle\cong \Z_p\oplus \Z_p\oplus C_p.
 \end{split} \]

Hence $V^{ab} \cong \Z_p\oplus \Z_p\oplus \Z_p\oplus C_p.$

\end{proof}

\begin{proof}[\textbf{Proof of Theorem~\ref{thm:main}}]

$(i) \implies (ii).$ We follow the idea of \cite[Theorem]{Dr87}. 
We will use induction on the size of $\Gamma$. If $\Gamma$ consists of a single vertex, the result is clear. Suppose that $\Gamma$ has more than one vertex. Let $\Gamma_1$, ..., $\Gamma_k$ be the connected components of $\Gamma$. If $k \geq 2$, then we have the free pro-$p$ decomposition $G_{\Gamma} = G_{\Gamma_1} \amalg  \cdots \amalg G_{\Gamma_k}$. Since  $\Gamma_i$, for each $1 \leq i \leq k$,  does not have an induced subgraph of type $C_4$ or $L_3$, by induction every closed subgroup of  $G_{\Gamma_i}$ is a right-angled Artin pro-$p$ group. Let $H$ be a closed subgroup of  $G_{\Gamma}$. Since $G_{\Gamma}$ is finitely generated, it follows that $H$ is a second-countable subgroup of $G_{\Gamma}$. The pro-$p$ version of the Kurosh subgroup theorem (\cite[Theorem 9.6.2]{Ribes-Book}) together with the fact that a free pro-$p$ product of right-angled Artin pro-$p$ groups is a right-angled Artin pro-$p$ group (as it follows from the pro-$p$ presentation), yield that $H$ is  a right-angled Artin pro-$p$ group. 

Now suppose that   $\Gamma$  is connected. Since  $\Gamma$  does not have an induced subgraph of type $C_4$ or $L_3$, by  \cite[Lemma]{Dr87},  $\Gamma$  has a vertex $v$ which is joined to every other vertex of  $\Gamma$. Let  $\Gamma'$ be the graph obtained from  $\Gamma$  by removing the vertex $v$ and all edges incident to $v$.  Then $G_{\Gamma} = \langle v \rangle \times G_{\Gamma'}$, where $\langle v \rangle$ is isomorphic to $\Z_p$. Let $ \varphi: G_{\Gamma}  \twoheadrightarrow G_{\Gamma'}$ be the natural projection. We have the short exact sequence 
$$\langle v \rangle  \hookrightarrow G_{\Gamma} \mathop{\twoheadrightarrow}^{\varphi}  G_{\Gamma'}.$$
Let $H$ be a closed subgroup of $G_{\Gamma}$. Then we have the following short exact sequence
$$H \cap \langle v \rangle  \hookrightarrow H \mathop{\twoheadrightarrow}^{\varphi}  \varphi(H).$$

We claim that this sequence splits. Note that, since $\Gamma'$ does not have an induced subgraph of type $C_4$ or $L_3$, by induction $\varphi(H)$ is a right-angled Artin pro-$p$ group. Let $ \Delta$ be a profinite graph such that $\varphi(H) = G_{\Delta}.$  Let $u$ be a vertex of $ \Delta$, and choose  $h \in H$  such that $\varphi(h) = u$. Note that $\varphi$ is an endomorphism and so $\varphi(h) = u\in G$. Then $\varphi(uh^{-1}) = uu^{-1} = 1$, so  $uh^{-1} \in  \langle v \rangle$. Hence, there is $\alpha_u \in  \Z_p$ such that $uv^{\alpha_u} \in H$. It follows that the function $\rho : \varphi(H) \to H$ defined by $\rho(u) = uv^{\alpha_u} $ for each $u \in \Delta$ is a homomorphism. Indeed, since $v$ is in the center of $G_{\Gamma}$, for $u_1, u_2 \in \Gamma'$, we have
$$[\rho(u_1), \rho(u_2)] = [u_1v^{\alpha_{u_1}}, u_2v^{\alpha_{u_2}} ] = [u_1, u_2] = 1. $$
Moreover, $ \varphi \rho = 1_{ \varphi(H)}$, i.e., $ \rho$ is a section.  

Now since $H \cap \langle v \rangle$ is in the center of $H$,  we have $H = (H \cap \langle v \rangle)  \times  \varphi(H)$. If $H \cap \langle v \rangle$ is trivial, then $H = G_{\Delta}$. Otherwise $H \cap \langle v \rangle$ is isomorphic to $\Z_p$, so we have $H = G_{\Delta'}$, where $\Delta'$ is a profinite graph obtained from $\Delta$ by adding a new vertex $w$ and a profinite space $E$ homeomorphic to $V(\Delta)$ of new edges that connect $w$ with all the vertices of $\Delta$. More precisely, if $\rho:E\longrightarrow V(\Delta)$ is a homeomorphism, then we define $d_0(e)=w$ and $d_1(e)=\rho(e)$ for each $e\in E$.

\medskip Note that the proof above yields also the implication
$(i) \implies (vi).$  

\medskip
$(ii) \implies (iii).$  Follows from the fact that every right-angled Artin pro-$p$ group has a torsion free abelianization, as  follows from its pro-$p$ presentation.

$(iii) \implies (i)$ and $(iv) \implies (i).$   Suppose that $\Gamma$ contains $C_4$ as an induced subgraph. Then  $G_{\Gamma}$ contains a subgroup isomorphic to $G_{C_4}$. Note that $G_{C_4}  = F_2 \times F_2$, where $F_2$ is a free pro-$p$ group of rank 2. Hence, by \cite[Theorem~5.6]{Qu14},  $G_{C_4}$  is not a Bloch-Kato pro-$p$ group, and therefore, $G_{\Gamma}$ is not as well. Moreover,  $G_{\Gamma}$ is not absolutely torsion free by  \cite[Proposition 4]{Wu85}.  On the other hand, if $\Gamma$ contains $L_3$ as an induced subgraph, then $G_{\Gamma}$  is not a  Bloch-Kato pro-$p$ group by Theorem  \ref{non-quadratic}, and  $G_{\Gamma}$ is not absolutely torsion free by Theorem \ref{absolutely-torsion-free}.

$(i) \implies (v).$ This implication follows from  \cite[Proposition 5.8]{CaQu19}.

$(v) \implies (iv).$ Follows from the solution of the Bloch-Kato conjecture (cf. \cite[Theorem A]{QuWe18}).

$(vi) \implies (v).$  This implication is proved implicitly in the proof of  \cite[Proposition 5.8]{CaQu19}. Indeed, it follows from the following two facts. If $K$ is a field that contains a primitive $p$th root of unity, then $G_{K((t))}(p)  \cong  \Z_p  \times  G_K(p)$. Moreover, if the pro-$p$ groups $G_1$ and $G_2$ occur as maximal pro-$p$ Galois groups, then their free pro-$p$ product occurs as well as the maximal pro-$p$ Galois group of some field $K$ (cf.  \cite[Remark 3.4]{Ef98}).
\end{proof}

 \section{Smoothness Conjecture}

First we recall some definitions and facts that will be needed in the proof of Theorem~\ref{smoothness-thm}.  A morphism $ \varphi: (G_1,  {\theta}_1)  \to (G_2, { \theta}_2)$ of $p$-oriented pro-$p$ groups is a continuous homomorphism $ \varphi: G_1  \to G_2$ such that $ \theta_1 =  \theta_2 \circ  \varphi.$ If $ \varphi$ induces an isomorphism $G_1/{ \Phi(G_1)}  \to G_2/{ \Phi(G_2)}$ on the Frattini quotients, then we say that $ \varphi$ is a cover.  Given a $p$-oriented pro-$p$ group $ \mathcal{G} = (G,   \theta),$  let 
$$K( \mathcal{G}) =  \langle h^{- \theta(g)}ghg^{-1} ~  \mid ~ g  \in G,  h  \in  \textrm{Ker}( \theta)\rangle.$$ 
Note that $K( \mathcal{G})$ is a normal subgroup of $G$ contained in $ \textrm{Ker}( \theta)$ and $ {\textrm{Ker}( \theta)}/{K( \mathcal{G})}$ is abelian. Moreover, since the generators of $K( \mathcal{G})$ belong to $ \textrm{Frat}(G)$, we have that $K( \mathcal{G}) \subseteq  \textrm{Frat}(G),$ so the morphism $ \mathcal{G} \to \mathcal{G}/{K( \mathcal{G})}$ is a cover. When $G$ is torsion free, $G/{ \textrm{Ker}( \theta)}  \cong  \textrm{im}( \theta)$ is isomorphic to either $ \mathbb{Z}_p$ or $ \{ 1  \}.$ Hence, the epimorphism $G/ {K( \mathcal{G})}  \to G/{ \textrm{Ker}( \theta)}$ splits, so we have 
$$ G/ {K( \mathcal{G})}   \cong (\textrm{Ker}( \theta)/ {K( \mathcal{G})})   \rtimes  ( G/{ \textrm{Ker}( \theta)}),$$
where the action is given by exponentiation by $\theta,$ i.e., $\bar{g}\bar{h}{\bar{g}}^{-1} =  {\bar{h}}^{\theta(g)}$ for $h \in \textrm{Ker}(\theta)$ and $g \in G.$
We say that $ \mathcal{G} = (G,   \theta)$ is Kummerian if $\textrm{Ker}( \theta)/ {K( \mathcal{G})}$ is a free abelian pro-$p$ group (cf. \cite[Definition3.4 and Proposition~3.3]{EfQu19}). Note that in this case $G/ {K( \mathcal{G})}$ is a herediteraly uniform pro-$p$ group (cf. \cite{KloSno14}  and \cite{Qu14}); in particular, if $\theta$ is the trivial homomorphism, then $G/ {K( \mathcal{G})}$ is a free abelian pro-$p$ group.

\begin{lem} \label{smooth-abs-tor-free}
Let $G$ be a pro-$p$ group and let $(G, 1)$ be the $p$-oriented pro-$p$ group $(G, \theta)$ with $\theta$ the trivial homomorphism. Then $(G, 1)$  is $1$-smooth if and only if $G$ is absolutely torsion free.   
\end{lem}
\begin{proof}
Note that, if $ \theta $  is the trivial homomorphism, then $\Z_p(1)$ is the trivial module $\Z_p$ (see (\ref{action})), so for every open subgroup $U$ of $G$,  $H^{1}_{cts}(U,\Z_p(1))$ coincides with the group of continuous group homomorphisms $\textrm{Hom}(U,  \Z_p).$  Since $\textrm{Hom}(U,  \Z_p) = U^{ \textrm{ab}}/  {\textrm{Tor}(U^{ \textrm{ab}})}$,  the natural epimorphism $\Z_p(1)  \twoheadrightarrow \F_p = {\Z_p(1)}/p{\Z_p(1)}$ induces an epimorphism 
  $H^{1}_{cts}(U,\Z_p(1)) \twoheadrightarrow H^{1}(U, {\Z_p(1)}/p{\Z_p(1)})$
  for every open subgroup $U$ of $G$ if and only if $U^{ \textrm{ab}}$ is torsion free. 
Hence,  $(G, 1)$  is $1$-smooth if and only if $G$ is absolutely torsion free. 
\end{proof}

\begin{lem} \label{non-smooth}
Let $\Gamma$ be a finite simplicial graph that contains as an induced subgraph either $C_4$ or $L_3$ and let $G = G_{\Gamma}.$  Then there is no homomorphism  $ \theta: G  \to  \Z_p^{ \times}$  such that  the $p$-oriented pro-$p$ group $(G,  \theta)$ is  $1$-smooth. 
\end{lem}
\begin{proof}
Suppose there is a homomorphism $ \theta: G  \to  \Z_p^{ \times}$  such that $(G,  \theta)$ is $1$-smooth.  For a closed subgroup $K$ of $G$, let $H^2(K, \theta_{\mid K}): = H^{2}(K, \Z_p(1))$; here we consider $(K,  \theta_{\mid K})$  as a $p$-oriented profinite group. Fix a closed subgroup $C$ of $G$ and let $ \mathcal{U}_C$ be the family of all open subgroups of $G$ containing $C$. By \cite[I.2.2, Proposition~8]{Se97}, $H^2(C, \theta_{\mid C}) \simeq \varinjlim_{U \in \mathcal{U}_C}H^2(U, \theta_{\mid U}).$ It follows that $(C,  \theta_{\mid C})$ is $1$-smooth for every closed subgroup $C$ of $G$. Hence, in order to prove the theorem, it suffices to show that there is no homomorphism  $ \theta: G  \to  \Z_p^{ \times}$  such that  the $p$-oriented pro-$p$ group $(G,  \theta)$ is  $1$-smooth when  $ \Gamma = C_4$ or $ \Gamma = L_3.$

First suppose that $\Gamma = C_4$.  Then $(G,  \theta)$ is $1$-smooth, and by  \cite[Theorem~7.1]{EfQu19}, every open subgroup of   $G$ is Kummerian.  By Lemma~\ref{smooth-abs-tor-free} and Theorem~\ref{thm:main}, $ \theta   \neq 1$. Hence, $ G/ {K( \mathcal{G})}   \cong (\textrm{Ker}( \theta)/ {K( \mathcal{G})})   \rtimes  ( G/{ \textrm{Ker}( \theta)})$ is a non-abelian hereditarily uniform pro-$p$ group. Recall that the morphism $ \varphi: G \to G/ {K( \mathcal{G})}$ is a cover.  Let $G = F_2  \times F_2 =   \langle x_1, x_2  \rangle  \times  \langle x_3, x_4 \rangle.$ Since $ \varphi$ is a cover, there is some $i$, say $i=1$, such that $ \varphi(x_1)  \not  \in \textrm{Ker}( \theta)/ {K( \mathcal{G})}.$ Then the centralizer of $ \varphi(x_1)$ is procyclic, which is a contradiction, since    $ \varphi(x_3)$ and  $ \varphi(x_4)$ commute with  $ \varphi(x_1).$

Now suppose that  $ \Gamma = L_3.$  Let $V = \langle yz, yz^x, y^p, t=wx\rangle$ be the  subgroup of $G = G_{L_3} = \langle x, y, z, w  ~  \mid ~  [x, y] =  [y, z] =  [z, w] = 1  \rangle$ defined in the proof of Theorem~\ref{absolutely-torsion-free}. Then $(V,  \theta_{\mid V})$ is $1$-smooth. Since the abelianization of $V$ has a non-trivial torsion,  $\theta_{\mid V}   \neq 1.$ Hence,  $ V/ {K( \mathcal{V})}   \cong (\textrm{Ker}( \theta_{\mid V} )/ {K( \mathcal{V})})   \rtimes  ( V/{ \textrm{Ker}( \theta_{\mid V} )})$ is a non-abelian hereditarily unfiorm pro-$p$ group. Moreover, since the minimal set of generators of $V$ consists of four elements and  $V/{ \textrm{Ker}( \theta_{\mid V} )}  \cong  \textrm{im}( \theta_{\mid V}) \cong \Z_p,$ we have that $\textrm{Ker}( \theta_{\mid V} )/ {K( \mathcal{V})}$ is a free abelian pro-$p$ group of rank 3. Consider the cover $ \psi: V \to V/ {K( \mathcal{V})},$ and let $a_1 = \psi(y^p)$, $a_2 = \psi(yz)$,  $a_3 = \psi(yz^x)$, $b = \psi(t).$ Then $ \{a_1, a_2, a_3, b  \}$ is a minimal generating set of $ V/ {K( \mathcal{V})} .$  Note that $y^p$ commutes with $yz$ and  $yz^x,$ so $a_1$ commutes with $a_2$ and $a_3.$ Since any element of  $V/ {K( \mathcal{V})}$ which is not contained in $\textrm{Ker}( \theta_{\mid V} )/ {K( \mathcal{V})}$ has a procyclic centralizer, it follows that $a_1, a_2, a_3$ are contained  in $\textrm{Ker}( \theta_{\mid V} )/ {K( \mathcal{V})},$ and moreover, they form a minimal generating set of the group. Hence, $b  \notin \textrm{Ker}( \theta_{\mid V} )/ {K( \mathcal{V})},$ so there is some $0  \neq  \gamma  \in  \Z_p,$ such that $b^{-1}ab = a^{1+ p \gamma}$ for every $a \in \textrm{Ker}( \theta_{\mid V} )/ {K( \mathcal{V})}.$ Now recall that  in $V$ we have the  relation ${(z^x)}^p={(z^p)}^t.$ If we apply $\psi$ to this relation, we obtain the relation 
$a_3^p{a_1}^{-1} = {a_2}^{p(1+ p\gamma)}{a_1}^{-(1+ p\gamma)},$ which clearly  yields a contradiction, since  $\langle a_1, a_2, a_3 \rangle$ is a free abelian pro-$p$ group of rank 3.

 Hence, there is no homomorphism  $ \theta: G  \to  \Z_p^{ \times}$  such that  the $p$-oriented pro-$p$ group $(G,  \theta)$ is  $1$-smooth.

\end{proof}

\begin{proof}[\textbf{Proof of Theorem~\ref{smoothness-thm}}]
The proof follows from Lemma~\ref{non-smooth},  Lemma~\ref{smooth-abs-tor-free} and Theorem~\ref{thm:main}. 
\end{proof}

 \section{Coherent right-angled Artin pro-$p$ groups}

A finite simplicial graph $\Gamma$ is called a \emph{chordal graph} if each circuit of $\Gamma$ of four or more
 vertices has a chord (i.e., an edge that is not part of the circuit but connects two vertices of the circuit). 
  Equivalently, every induced circuit in the graph should have exactly three vertices.

\begin{lem}\label{lemma-chordal}
If $\Gamma$ is a chordal graph,  then the commutator subgroup of the  right-angled Artin
 pro-$p$ group $ G_{\Gamma}$ is  a free pro-$p$ group. 
\end{lem}
\begin{proof}
First, let us make a useful observation. Given a finite simplicial graph $\Pi,$ denote by $X$  its set of vertices.  Then by definition, $G_{\Pi}$ is generated 
by the elements of $X$ and it is defined by the commutator relations coming from all pairs of adjacent vertices. 
Since in all these commutator relations the exponent sum of each letter of $X$ is 0, a word $w$ on $X \cup X^{-1}$ represents an element of $[G_{\Pi}, G_{\Pi}]$ if and only if 
the exponent sum of each letter of $X$ in $w$ is 0. It follows that, if $\Delta$ is an induced subgraph of $\Pi,$ then $[G_{\Delta}, G_{\Delta}] = [G_{\Pi}, G_{\Pi}] \cap G_{\Delta}.$

Now let $\Gamma$ be a chordal graph. We proceed by induction on the number of vertices of  $\Gamma.$ 
Clearly, the result holds when  $\Gamma$  consists of a single vertex. Therefore, let us assume that  $\Gamma$  has $n > 1$ vertices.
If $\Gamma$ is a complete graph, then  $ G_{\Gamma}$ is abelian, so $[G_{\Gamma},  G_{\Gamma}]$ is trivial.
Otherwise, there are proper induced subgraphs $\Gamma_1$ and $\Gamma_2$ of $\Gamma,$ 
such that $\Gamma = \Gamma_1 \cup \Gamma_2$ and $\Lambda: = \Gamma_1 \cap \Gamma_2$ is a complete (possibly empty)
subgraph of $\Gamma$ (\cite{Dir61}). Now it follows from the definition of right-angled Artin pro-$p$ groups that   $G_{\Gamma} = G_{\Gamma_1}\amalg_{G_{\Lambda}}G_{\Gamma_2}.$
Moreover, the above observation yields that 
 \[
[G_{\Gamma_1}, G_{\Gamma_1}] = [G_{\Gamma}, G_{\Gamma}] \cap G_{\Gamma_1},  ~  [G_{\Gamma_2}, G_{\Gamma_2}] = [G_{\Gamma}, G_{\Gamma}] \cap G_{\Gamma_2}
\]
\[
\text{and } [G_{\Gamma}, G_{\Gamma}] \cap G_{\Lambda} =  [G_{\Lambda}, G_{\Lambda}]  = 1.
\]
Since  $\Gamma_1$ and  $\Gamma_2$ are induced subgraphs,   they are both chordal graphs, so by induction, 
it follows that  $[G_{\Gamma_1}, G_{\Gamma_1}]$ and $[G_{\Gamma_2}, G_{\Gamma_2}] $ are free pro-$p$ groups.
By  \cite[Theorem~9.6.1]{Ribes-Book},  we have that  $[G_{\Gamma}, G_{\Gamma}]$ is a free pro-$p$ product of free pro-$p$ groups and therefore, it  is a free pro-$p$ group. 
\end{proof}

 \begin{cor} \label{cor-coherent} Let $\Gamma$ be a chordal graph. Then every finitely generated subgroup 
 of  $G = G_{\Gamma}$ is  of type $FP_{\infty}$. In particular, $G$ is coherent.
 \end{cor}
 
 \begin{proof} Let $H$ be a finitely generated subgroup of $G_\Gamma$.  Then $H$ is a free by abelian pro-$p$ group, since by Lemma \ref{lemma-chordal}, $G_{\Gamma}$ is a free by abelian pro-$p$ group.
 Let $F$ be a normal  free pro-$p$ subgroup of $H$ such that $H/F=Q$ is abelian. 
 By \cite[Theorem 2, Section 3]{King},  $H$ is of type $FP_{\infty}$ if and only if $H_i(F, \Z_p)$ is finitely generated $\Z_p[[Q]]$-module for each $i \geq 1$. 
   Now since $H$ is finitely generated, $H_1(F,\Z_p)$ is a finitely generated $\Z_p[[Q]]$-module. Moreover, since $F$ is free,  $H_i(F,\Z_p)=0$ for each  $i>1$.
   Hence, $H$ is  of type $FP_{\infty}.$ In particular, $H^2(G, \F_p)$ is finite, which yields that $G$ is coherent, by \cite[Proposition 27]{Se97}.
 \end{proof}

Before  proving the converse of Corollary \ref{cor-coherent},  we define some natural subgroups of right-angled Artin pro-$p$ groups, following an idea of Droms in  \cite{Dr87c}. Let $\Gamma$ be a finite simplicial graph with a vertex set $X.$ 
Given an element $g \in G_{\Gamma},$ with $g=x_1^{k_1}x_2^{k_2}\cdots x_n^{k_n}$, where each $x_i$ is a vertex of $X$,
 define $|g| = k_1 + k_2 + \cdots + k_n.$
Note that  $|g|$ is well defined, since each defining relation of $G_{\Gamma}$ has exponent sum 0. Then the map that sends $g$ to $|g|$ defines a homomorphism
 from the abstract  right-angled Artin group $G(\Gamma)$ to $\Z$, which we extend  to the natural homomorphism $f: G_\Gamma\to \Z_p$ on the corresponding pro-$p$ completions.  
 Let $K_{\Gamma}$ be the kernel of $f$. Observe that for any induced subgraph $\Delta$ of $\Gamma$,  one has $K_\Gamma \cap G_{\Delta}=K_{\Delta}$.

\begin{lem}\label{lemma-chordal} Let $\Gamma$ be a finite connected simplicial graph, and let $\Gamma_1$ and $\Gamma_2$ be proper induced subgraphs of $\Gamma$ such that  $\Gamma=\Gamma_1\cup \Gamma_2$ and $\Delta=\Gamma_1\cap \Gamma_2\neq \emptyset$. Then  $K_\Gamma=K_{\Gamma_1}\amalg_{K_{\Delta}} K_{\Gamma_2}$. 

\end{lem}

\begin{proof} First note  that $G_{\Gamma}=G_{\Gamma_1}\amalg_{G_{\Delta}} G_{\Gamma_2}$. Let $S=S(G_\Gamma)$  be the standard pro-$p$ tree associated with this  amalgamated free pro-$p$ product.   
By  \cite[Theorem~6.6.1]{Ribes-Book}, in order to prove the Lemma, it suffices to show that $ K_{\Gamma}$ acts transitively on $E(S),$ or equivalently, that $K_{\Gamma}G_{\Delta}=G_\Gamma$. 
Note that, since $\Delta$ is non-empty, it follows that $f_{|G_{\Delta}}: G_{\Delta} \to \Z_p$ is surjective. Hence $K_{\Gamma}G_{\Delta}=G_\Gamma$, as desired. 

\end{proof}

\begin{lem} \label{tree}
Let $T$ be a finite tree and let $G = G_{T}$ be the  right-angled Artin pro-$p$ group associated to $T$. Then   $K_T$ is a free pro-$p$ group of finite rank.  
\end{lem}
\begin{proof}
We proceed by induction on the number of vertices of  $T.$ Clearly, the result holds when  $T$  consists of a single vertex. 
Therefore, let us assume that  $T$  has $n > 1$ vertices. Since  $T$ is a finite tree, it contains a pending edge $e$ with a pending vertex $v$. 
Let $T_0$ be the tree obtained from $T$ by removing $e$ and $v$. Then $G_T= G_{T_0}\amalg_{\langle w\rangle} (\langle w\rangle \times \langle v\rangle)$, where  $w$ is the other vertex of $e$.
 By Lemma \ref{lemma-chordal}, $K_{T}=K_{T_0}\amalg_{K_{\{w\}}} K_{\{v,w, e\}}$. But ${K_{\{w\}}}$ is trivial, $K_{\{v, w, e\}}$ is infinite procyclic and $K_{T_0}$ is a free pro-$p$ of finite rank, by the induction hypothesis. 
 Hence,  $K_{T}=K_{T_0}\amalg K_{\{v,w, e\}}$ is a free pro-$p$  group of finite rank.
\end{proof}

 \begin{pro}\label{not-coherent} If a finite simplicial graph $\Gamma$ is not  chordal, then $G = G_{\Gamma}$ is not coherent.
 
 \end{pro}
 \begin{proof} Suppose $\Gamma$ is a circuit of length greater than 3 and let $v,w$ be two non-adjacent vertices of $\Gamma$. 
 Then there are proper induced subgraphs $T_1$, $T_2$ of  $\Gamma$ such that $\Gamma=T_1\cup T_2$, $T_1\cap T_2=\{v,w\}$ and  
 $T_1,T_2$ are trees, and so $G_\Gamma=G_{T_1}\amalg_{F(v,w)} G_{T_2}$ splits as an amalgamated free pro-$p$ product.   
 Note that $K_{\{v,w \}}$ is the normal closure of $v^{-1}w$ in $F(v,w)$; in particular,  $K_{\{v,w \}}$  is a free pro-$p$ group of infinite rank. 
  By Lemma \ref{lemma-chordal} and Lemma  \ref{tree},  $K_{\Gamma}= K_{T_1}\amalg_{K_{\{v,w \}}} K_{T_2}$, where $K_{T_1}$ and $K_{T_1}$  are free pro-$p$ groups of finite rank.

  \smallskip
   
Now by \cite[Prop.~9.2.13]{RZ10}, we have the following Mayer-Vietoris sequence in cohomology

\[
 \xymatrix@C=0.5truecm{ 0\ar[r] & H^1( K_{\Gamma},\F_p)\ar[r]& H^1( K_{T_1},\F_p)\oplus H^1( K_{T_2},\F_p)\ar[r]& H^{1}( K_{\{v,w \}}, \F_p)  \ar`r[d]`[l] `[dlll] `[dll] [dll]  \\
&H^2( K_{\Gamma}, \F_p)\ar[r]& H^2( K_{T_1},\F_p)\oplus H^2( K_{T_2},\F_p)\ar[r]&\cdots}
\]

   Then $H^1(K_{T_1},  \F_p)\oplus H^1(K_{T_2},  \F_p)$ is finite and  $H^2(K_{T_1},  \F_p)\oplus H^2(K_{T_2},  \F_p)=0$. 
 On the other hand,  $H^1(K_{\{v,w \}},  \F_p)$ is infinite, since $K_{\{v,w \}}$ is a free pro-$p$ group
  of infinite rank. Thus $H^2(K_{\Gamma},  \F_p)$ is infinite, or equivalently, $K_{\Gamma}$ is not finitely presented (see  \cite[Proposition 27]{Se97}). Hence, $G$ is not coherent.
 \end{proof} 
 
 \begin{proof}[\textbf{Proof of Theorem~\ref{coherence}}]
The proof follows from Corollary~\ref{cor-coherent} and Proposition~\ref{not-coherent}. 
\end{proof}

\ackn 
The first author acknowledges support from the Alexander von Humboldt Foundation, CAPES (grant 88881.145624/2017-01), CNPq and FAPERJ. 
The second  author is partially supported by FAPDF and CNPq.  The first author is grateful to Claudio Quadrelli and Thomas Weigel for telling him about Conjecture 1.1 and for inspiring discussions.
This work started during the workshop ``Group theory days in D\"usseldorf 2020''; the authors thank the Heinrich Heine University in D\"usseldorf for its hospitality.
\bibliographystyle{plain}

\end{document}